\documentclass[]{interact}

\usepackage{epstopdf}
\usepackage[caption=false]{subfig}
\usepackage{xcolor}
\usepackage{amsmath}
\usepackage{amsmath,cases}
\usepackage{mathtools}
\usepackage{xcolor}
\usepackage[numbers,sort&compress]{natbib}
\usepackage{mathtools}
\bibpunct[, ]{[}{]}{,}{n}{,}{,}
\renewcommand\bibfont{\fontsize{10}{12}\selectfont}
\makeatletter
\def\NAT@def@citea{\def\@citea{\NAT@separator}}
\makeatother

\theoremstyle{plain}
\newtheorem{theorem}{Theorem}[section]
\newtheorem{lemma}[theorem]{Lemma}

\newtheorem{proposition}[theorem]{Proposition}

\theoremstyle{definition}
\newtheorem{definition}[theorem]{Definition}
\newtheorem{example}[theorem]{Example}

\theoremstyle{remark}

\DeclareMathOperator{\sign}{sign}

\begin{document}

\articletype{Research Paper}

\title{A Topological Approach to Singular Double-Phase Equations with Variable Exponents}

\author{\name{Mustafa Avci\thanks{CONTACT M.~Avci. Email:  mavci@athabascau.ca (primary) \& avcixmustafa@gmail.com}}
\affil{Faculty of Science and Technology, Applied Mathematics, Athabasca University, AB, Canada}}

\maketitle

\begin{abstract}
In the present paper, we study a singular double phase variable exponent Dirichlet problem in the setting of a new Musielak-Orlicz Sobolev space with the nonlinearity (the external source) having gradient dependence (so-called convection terms). We apply a topological existence result incorporating the Leray–Schauder degree and homotopy mapping together to prove the existence of at least one nontrivial solution.
\end{abstract}


\begin{amscode}
35A01; 35A16; 35D30; 35J75
\end{amscode}

\section{Introduction}
In this article, we study the following double phase singular problem

\begin{equation}\label{e1.1}
\begin{cases}
\begin{array}{rlll}
&-\mathrm{div}(|\nabla u|^{p(x)-2}\nabla u+\mu(x)|\nabla u|^{q(x)-2}\nabla u)\\
&+|x|^{-p(x)}|u|^{p(x)-2}u+|x|^{-q(x)}\mu(x)|u|^{q(x)-2}u\\
&=f(x,u,\nabla u) \text{ in }\Omega, \\
&u=0  \qquad \qquad \qquad \qquad \qquad \quad \quad \text{ on }\partial \Omega,\tag{$\mathcal{P}$}
\end{array}
\end{cases}
\end{equation}
where $\Omega$ is a bounded domain in $\mathbb{R}^N$ $(N\geq 3)$ with Lipschitz boundary; $p,q \in C(\overline{\Omega })$; $f$ is a Carathéodory function; and $0\leq \mu(\cdot)\in L^\infty(\Omega)$.\\

The equations of type (\ref{e1.1}) are found in models of materials with non-uniform (anisotropic) properties, fluid mechanics, image processing, and elasticity with non-homogeneous materials.
Thus the problem (\ref{e1.1}) could be used to model various real-world phenomena mostly because of the existence of the operator
\begin{equation}\label{e1.2a}
\mathrm{div}\left(|\nabla u|^{p(x)-2}\nabla u+\mu(x)|\nabla u|^{q(x)-2}\nabla u\right)
\end{equation}
 which governs anisotropic and heterogeneous diffusion associated with the energy functional
\begin{equation}\label{e1.2b}
u\to \int_{\Omega}\left(\frac{|\nabla u|^{p(x)}}{p(x)}+\mu(x)\frac{|\nabla u|^{q(x)}}{q(x)}\right)dx,\,\ u \in W_0^{1,\mathcal{H}}(\Omega)
\end{equation}
is called a "double phase" operator because it encapsulates two different types of elliptic behavior within the same framework. The functional of type (\ref{e1.2b}) was introduced in \cite{zhikov1987averaging} for the constant exponents. Subsequently, many studies have been devoted in this direction (see, e.g. \cite{baroni2015harnack,baroni2018regularity,colombo2015bounded,colombo2015regularity,marcellini1991regularity,marcellini1989regularity}) due to its applicability into different disciplines. The nonlinearity $f(x,u,\nabla u)$, on the other hand, captures the effects of convection, which combines advection with nonlinear diffusion or reaction phenomena, depending on the system.

We would like to mention that, to our best knowledge, there have been few papers to handle the equations involving variable exponent double phase operator with convection term where the main method used was monotone operator theory. We refer the interested reader to the studies \cite{bahrouni2019double,cencelj2018double,galewski2024variational,gasinski2020existence,bu2022p,lapa2015no,el2023existence,albalawi2022gradient,motreanu2020quasilinear} and the references there in.

Primarily inspired by the cited works, we study a double phase  variable exponent singular problem with a convection term. However, a key challenge and innovation of this article lies in our consideration of a distance-dependent singularity $\frac{|u|^{p(x)-2}u}{|x|^{p(x)}}+\mu(x)\frac{|u|^{q(x)-2}u}{|x|^{q(x)}}$, i.e. a Hardy-type potential, which makes the model we suggest more applicable since it could additionally model localized phenomena (e.g., near wells, defects, or critical points). As well, we use a new topological existence result \cite[Theorem 4.1]{avci2019topological} incorporated with the Leray–Schauder degree and homotopy mapping to obtain the existence of at least one nontrivial solution to problem (\ref{e1.1}), which would be another innovation presented in this article.\\

The paper is organised as follows. In Section 2, we first provide some background for the theory of variable Sobolev spaces $W_{0}^{1,p(\cdot)}(\Omega)$ and the Musielak-Orlicz Sobolev space $W_0^{1,\mathcal{H}}(\Omega)$, and then present a crucial auxiliary result: a Hardy-type inequality. In Section 3, we set up the nonlinear operator equations corresponding to problem (\ref{e1.1}), and show the existence
of weak solutions to the considered problem.

\section{Mathematical Background and Auxiliary Results}

We start with some basic concepts of variable Lebesgue-Sobolev spaces. For more details, and the proof of the following propositions, we refer the reader to \cite{cruz2013variable,diening2011lebesgue,edmunds2000sobolev,fan2001spaces,radulescu2015partial}.
\begin{equation*}
C_{+}\left( \overline{\Omega }\right) =\left\{h\in C(\overline{\Omega }):\, h(x)>1 \text{ for all\ }x\in \overline{\Omega }\right\} .
\end{equation*}
For $h\in C_{+}( \overline{\Omega }) $ denote
\begin{equation*}
h^{-}:=\underset{x\in \overline{\Omega }}{\min }h( x), \quad  h^{+}:=\underset{x\in \overline{\Omega }}{\max}h(x) <\infty .
\end{equation*}
For any $h\in C_{+}\left( \overline{\Omega }\right) $, we define \textit{the variable exponent Lebesgue space} by
\begin{equation*}
L^{h(x)}(\Omega) =\left\{ u\mid u:\Omega\rightarrow\mathbb{R}\text{ is measurable},\int_{\Omega }|u(x)|^{h(x)}dx<\infty \right\}.
\end{equation*}
Then, $L^{h(x)}(\Omega)$ endowed with the norm
\begin{equation*}
\|u\|_{h(x)}=\inf \left\{ \lambda>0:\int_{\Omega }\left\vert \frac{u(x)}{\lambda }
\right\vert ^{h(x)}dx\leq 1\right\} ,
\end{equation*}
becomes a Banach space.\\
The convex functional $\rho :L^{h(x) }(\Omega) \rightarrow\mathbb{R}$ defined by
\begin{equation*}
\rho(u) =\int_{\Omega }|u(x)|^{h(x)}dx,
\end{equation*}
is called modular on $L^{h(x)}(\Omega)$.

\begin{proposition}\label{Prop:2.2} If $u,u_{n}\in L^{h(x)}(\Omega)$, we have
\begin{itemize}
\item[$(i)$] $\|u\|_{h(x)}<1 ( =1;>1) \Leftrightarrow \rho(u) <1 (=1;>1);$
\item[$(ii)$] $\|u\|_{h(x) }>1 \implies |u|_{h(x)}^{h^{-}}\leq \rho(u) \leq |u|_{h(x) }^{h^{+}}$;\newline
$|u|_{h(x)}\leq1 \implies \|u\|_{h(x)}^{h^{+}}\leq \rho(u) \leq \|u\|_{h(x)}^{h^{-}};$
\item[$(iii)$] $\lim\limits_{n\rightarrow \infty }\|u_{n}-u\|_{h(x)}=0\Leftrightarrow \lim\limits_{n\rightarrow \infty }\rho (u_{n}-u)=0$.
\end{itemize}
\end{proposition}

\begin{proposition}\label{Prop:2.2bb}
Let $h_1(x)$ and $h_2(x)$ be measurable functions such that $h_1\in L^{\infty}(\Omega )$ and $1\leq h_1(x)h_2(x)\leq \infty$ for a.e. $x\in \Omega$. Let $u\in L^{h_2(x)}(\Omega ),~u\neq 0$. Then
\begin{itemize}
\item[$(i)$] $\|u\|_{h_1(x)h_2(x)}\leq 1\text{\ }\Longrightarrow
\|u\|_{h_1(x)h_2(x)}^{h_1^{+}}\leq \left\||u|^{h_1(x)}\right\|_{h_2(x)}\leq \|u\|_{h_1(x)h_2(x)}^{h^{-}}$
\item[$(ii)$] $\|u\|_{h_1(x)h_2(x)}> 1\ \Longrightarrow \|u\|_{h_1(x)h_2(x)}^{h_1^{-}}\leq \left\||u|^{h_1(x)}\right\|_{h_2(x)}\leq \|u\|_{h_1(x)h_2(x)}^{h_1^{+}}$
\item[$(iii)$] In particular, if $h_1(x)=h$ is constant then
\begin{equation*}
\||u|^{h}\|_{h_2(x)}=\|u\|_{h_2(x)h}^{h}.
\end{equation*}
\end{itemize}
\end{proposition}
\noindent The variable exponent Sobolev space $W^{1,h(x)}( \Omega)$ is defined by
\begin{equation*}
W^{1,h(x)}( \Omega) =\{u\in L^{h(x)}(\Omega) : |\nabla u| \in L^{h(x)}(\Omega)\},
\end{equation*}
with the norm
\begin{equation*}
\|u\|_{1,h(x)}=\|u\|_{h(x)}+\|\nabla u\|_{h(x)}, \quad \forall u\in W^{1,h(x)}(\Omega),
\end{equation*}
where $\|\nabla u\|_{h(x)}=\|\,|\nabla u|\,\|_{h(x)}$ due to the fact that $u$ is real-valued and
$$
|\nabla |u||=|\nabla u|=|\sign u\, \nabla u|\quad \text{a.e.}
$$
\begin{proposition}\label{Prop:2.4} If $1<h^{-}\leq h^{+}<\infty $, then the spaces $L^{h(x)}( \Omega)$ and $W^{1,h(x)}(\Omega)$ are separable and reflexive Banach spaces.
\end{proposition}
\noindent The zero trace space $W_{0}^{1,h(x)}(\Omega)$ is defined as
\begin{equation*}
W_{0}^{1,h(x)}(\Omega):=\bigcap_{C_{0}^{\infty }(\Omega )\subset \mathcal{K}}\mathcal{K}\quad \textit { for closed subsets } \mathcal{K}\subset W^{1,h(x)}(\Omega).
\end{equation*}
With this definition, $W_{0}^{1,h(x)}(\Omega)$ is also a separable and reflexive Banach space. Furthermore, since Poincar\'{e} inequality holds on $W_{0}^{1,h(x)}(\Omega)$, we have
\begin{equation*}
\|\nabla u\|_{h(x)}\leq \|u\|_{1,h(x)}\leq (C+1)\|\nabla u\|_{h(x)},\quad \forall u\in W_{0}^{1,h(x)}(\Omega),\quad 0<C \in \mathbb{R},
\end{equation*}
which implies $\|u\|_{1,h(x)}$ and $\|\nabla u\|_{h(x)}$ are equivalent norms on $W_{0}^{1,h(x)}(\Omega)$. Therefore, for any $u\in W_{0}^{1,h(x)}(\Omega)$ we can define an equivalent norm $\|u\|$ such that
\begin{equation*}
\|u\| =\|\nabla u \|_{h(x)}.
\end{equation*}

\begin{proposition}\label{Prop:2.5} Let $r\in C(\overline{\Omega })$. If $1\leq r(x) <h^{\ast }(x)$ for all $x\in
\overline{\Omega }$, then the embeddings $W^{1,h(x)}(\Omega) \hookrightarrow L^{r(x)}(\Omega)$ and $W_0^{1,h(x)}(\Omega) \hookrightarrow L^{r(x)}(\Omega)$  are compact and continuous, where
$h^{\ast }( x) =\left\{\begin{array}{cc}
\frac{Nh(x) }{N-h( x) } & \text{if }h(x)<N, \\
+\infty & \text{if }h( x) \geq N.
\end{array}
\right. $
\end{proposition}

Throughout the paper, we assume the following.
\begin{itemize}
\item[$(H_1)$] $p,q\in C(\overline{\Omega})$, $1<p(x),q(x)<N$ with $p(x)<q(x)<p^*(x)$ for all $x\in
\overline{\Omega }$.
\item[$(H_2)$] $\mu\in L^\infty(\Omega)$ such that $\mu(\cdot)\geq 0$.
\end{itemize}

In order to deal with the problem (\ref{e1.1}), we need some theory of the Musielak-Orlicz Sobolev space $W_0^{1,\mathcal{H}}(\Omega)$. Thus,
in the sequel, we introduce the double phase operator, the Musielak–Orlicz space, and the Musielak–Orlicz Sobolev space, respectively.\\

Let $\mathcal{H}:\Omega\times [0,\infty]\to [0,\infty]$ be the nonlinear function, i.e. the \textit{double phase operator}, defined by
\[
\mathcal{H}(x,t)=t^{p(x)}+\mu(x)t^{q(x)}\ \text{for all}\ (x,t)\in \Omega\times [0,\infty].
\]
Then the corresponding modular $\rho_\mathcal{H}(\cdot)$ is given by
\[
\displaystyle\rho_\mathcal{H}(u)=\int_\Omega\mathcal{H}(x,|u|)dx=
\int_\Omega\left(|u|^{p(x)}+\mu(x)|u|^{q(x)}\right)dx.
\]
The \textit{Musielak-Orlicz space} $L^{\mathcal{H}}(\Omega)$, is defined by
\[
L^{\mathcal{H}}(\Omega)=\left\{u:\Omega\to \mathbb{R}\,\, \text{measurable};\,\, \rho_{\mathcal{H}}(u)<+\infty\right\},
\]
endowed with the Luxemburg norm
\[
\|u\|_{\mathcal{H}}=\inf\left\{\zeta>0: \rho_{\mathcal{H}}\left(\frac{u}{\zeta}\right)\leq 1\right\}.
\]
Analogous to Proposition \ref{Prop:2.2}, there are similar relationship between the modular $\rho_{\mathcal{H}}(\cdot)$ and the norm $\|\cdot\|_{\mathcal{H}}$, see \cite[Proposition 2.13]{crespo2022new} for a detailed proof.

\begin{proposition}\label{Prop:2.2a}
Assume $(H_1)$ hold, and $u\in L^{\mathcal{H}}(\Omega)$. Then
\begin{itemize}
\item[$(i)$] If $u\neq 0$, then $\|u\|_{\mathcal{H}}=\lambda\Leftrightarrow \rho_{\mathcal{H}}(\frac{u}{\zeta})=1$,
\item[$(ii)$] $\|u\|_{\mathcal{H}}<1\ (\text{resp.}\ >1, =1)\Leftrightarrow \rho_{\mathcal{H}}(\frac{u}{\zeta})<1\ (\text{resp.}\ >1, =1)$,
\item[$(iii)$] If $\|u\|_{\mathcal{H}}<1\Rightarrow \|u\|_{\mathcal{H}}^{q^+}\leq \rho_{\mathcal{H}}(u)\leq \|u\|_{\mathcal{H}}^{p^-}$,
\item[$(iv)$]If $\|u\|_{\mathcal{H}}>1\Rightarrow \|u\|_{\mathcal{H}}^{p^-}\leq \rho_{\mathcal{H}}(u)\leq \|u\|_{\mathcal{H}}^{q^+}$,
\item[$(v)$] $\|u\|_{\mathcal{H}}\to 0\Leftrightarrow \rho_{\mathcal{H}}(u)\to 0$,
\item[$(vi)$]$\|u\|_{\mathcal{H}}\to +\infty\Leftrightarrow \rho_{\mathcal{H}}(u)\to +\infty$,
\item[$(vii)$] $\|u\|_{\mathcal{H}}\to 1\Leftrightarrow \rho_{\mathcal{H}}(u)\to 1$,
\item[$(viii)$] If $u_n\to u$ in $L^{\mathcal{H}}(\Omega)$, then $\rho_{\mathcal{H}}(u_n)\to\rho_{\mathcal{H}}(u)$.
\end{itemize}
\end{proposition}
The \textit{Musielak-Orlicz Sobolev space} $W^{1,\mathcal{H}}(\Omega)$ is defined by
\[
W^{1,\mathcal{H}}(\Omega)=\left\{u\in L^{\mathcal{H}}(\Omega):
|\nabla u|\in L^{\mathcal{H}}(\Omega)\right\},
\]
and equipped with the norm
\[
\|u\|_{1,\mathcal{H}}=\|u\|_{\mathcal{H}}+\|\nabla u\|_{\mathcal{H}},
\]
where $\|\nabla u\|_{\mathcal{H}}=\|\,|\nabla u|\,\|_{\mathcal{H}}$.\\
The space $W_0^{1,\mathcal{H}}(\Omega)$ is defined as $\overline{C_{0}^{\infty }(\Omega )}^{\|\cdot\|_{1,\mathcal{H}}}=W_0^{1,\mathcal{H}}(\Omega)$. Notice that $L^{\mathcal{H}}(\Omega), W^{1,\mathcal{H}}(\Omega)$ and $W_0^{1,\mathcal{H}}(\Omega)$ are uniformly convex and reflexive Banach spaces. Moreover, we have the following embeddings \cite[Proposition 2.16]{crespo2022new}.
\begin{proposition}\label{Prop:2.7a}
Let $(H)$ be satisfied. Then the following embeddings hold:
\begin{itemize}
\item[$(i)$] $L^{\mathcal{H}}(\Omega)\hookrightarrow L^{h(x)}(\Omega), W^{1,\mathcal{H}}(\Omega)\hookrightarrow W^{1,h(x)}(\Omega)$, $W_0^{1,\mathcal{H}}(\Omega)\hookrightarrow W_0^{1,h(x)}(\Omega)$ are continuous for all $h\in C(\overline{\Omega})$ with $1\leq h(x)\leq p(x)$ for all $x\in \overline{\Omega}$.
\item[$(ii)$] $W^{1,\mathcal{H}}(\Omega)\hookrightarrow L^{h(x)}(\Omega)$ and $W_0^{1,\mathcal{H}}(\Omega)\hookrightarrow L^{h(x)}(\Omega)$ are compact for all $h\in C(\overline{\Omega})$ with $1\leq h(x)< p^*(x)$ for all $x\in \overline{\Omega}$.
\end{itemize}
\end{proposition}
\noindent By Proposition \ref{Prop:2.7a}, we have the continuous embedding $W_0^{1,\mathcal{H}}(\Omega)\hookrightarrow L^{h(x)}(\Omega)$
\[
\|u\|_{h(x)}\leq c_{\mathcal{H}}\|u\|_{1,\mathcal{H},0},
\]
where $c_{\mathcal{H}}$ denotes the best embedding constant. Moreover, by \cite[Proposition 2.18]{crespo2022new}, $W_0^{1,\mathcal{H}}(\Omega)$ is compactly embedded in $L^{\mathcal{H}}(\Omega)$. Thus,
$W_0^{1,\mathcal{H}}(\Omega)$ can be equipped with the equivalent norm
\[
\|u\|_{1,\mathcal{H},0}=\|\nabla u\|_{\mathcal{H}}.
\]
\begin{proposition}\label{Prop:2.7}
For the convex functional
$$\varrho_{\mathcal{H}}(u):=\int_\Omega\left(\frac{|\nabla u|^{p(x)}}{p(x)}+\mu(x)\frac{|\nabla u|^{q(x)}}{q(x)}\right)dx,
$$
we have $\varrho_{\mathcal{H}} \in C^{1}(W_0^{1,\mathcal{H}}(\Omega),\mathbb{R})$ with the derivative
$$
\langle\varrho^{\prime}_{\mathcal{H}}(u),\varphi\rangle=\int_{\Omega}(|\nabla u|^{p(x)-2}\nabla u+\mu(x)|\nabla u|^{q(x)-2}\nabla u)\cdot\nabla \varphi dx,
$$
for all  $u, \varphi \in W_0^{1,\mathcal{H}}(\Omega)$, where $\langle \cdot, \cdot\rangle$ is the dual pairing between $W_0^{1,\mathcal{H}}(\Omega)$ and its dual $W_0^{1,\mathcal{H}}(\Omega)^{*}$ \cite{crespo2022new}.
\end{proposition}

Next, we present a Hardy-type inequality that allows us to control the singular potential. We also note that our result \eqref{e3.11mm} extends the constant exponent case given by Mitidier \cite{mitidieri2000simple}
\begin{equation}\label{e3.11man}
\int_{\Omega}\frac{|u(x)|^{s}}{|x|^{s}}dx \leq \left(\frac{s}{N-s}\right)^{s} \int_{\Omega}|\nabla u(x)|^{s}dx.
\end{equation}

\begin{lemma}\label{Lem:2.1} Assume $(H_1)$ holds.\\
$(i)$ There exists a positive constant $C_{N}(p,q)$ such that the inequality
\begin{equation}\label{e3.11mm}
\int_{\Omega}\left(\frac{|u|^{p(x)}}{p(x)|x|^{p(x)}}+\mu(x)\frac{|u|^{q(x)}}{q(x)|x|^{q(x)}}\right)dx \leq C_{N}(p,q)\|u\|_{1,\mathcal{H},0}^{\kappa},\,\,\ \forall u \in W_0^{1,\mathcal{H}}(\Omega)
\end{equation}
holds, where $C_{N}(p,q):=(p^-)^{-1}\hat{c}_M(1+\|\mu\|_{\infty})\max\{C_{N}(q^+),C_{N}(p^-)\}$ with
\newline $C_{N}(q^+):=\left(\frac{q^+}{N-q^+}\right)^{q^+}$, $C_{N}(p^-):=\left(\frac{p^-}{N-p^-}\right)^{p^-}$;
\newline $\hat{c}_M:=\max\{\hat{c}_1,\hat{c}_2\}$ with $\hat{c}_1>0,\hat{c}_2>0$ are the embedding constants not depending on $u$; and $\kappa=p^-$ if $\|u\|_{1,\mathcal{H},0}<1$,
$\kappa=q^{+}$ if $\|u\|_{1,\mathcal{H},0}\geq1$.\\
$(ii)$ There exists $0<M \in \mathbb{R}$ such that the inequality
\begin{equation}\label{e3.11mn}
\int_{\Omega}\left(\frac{|u|^{p(x)}}{|x|^{p(x)}}+\mu(x)\frac{|u|^{q(x)}}{|x|^{q(x)}}\right)dx \geq \frac{1}{M^\tau}\|u\|^{\kappa}_{\mathcal{H}},\,\,\ \forall u \in W_0^{1,\mathcal{H}}(\Omega)
\end{equation}
holds; that is, the real number $\frac{1}{M^{\tau}}\|u\|^{\kappa}_{\mathcal{H}}$ is a lower bound for $\int_{\Omega}\left(\frac{|u|^{p(x)}}{|x|^{p(x)}}+\mu(x)\frac{|u|^{q(x)}}{|x|^{q(x)}}\right)dx$, where $\tau= p^-$ if $0<|x|\leq M<1$, and $\tau= q^+$ if $1\leq |x|\leq M$.
\end{lemma}

\begin{proof} The proof follows the same approach with a similar reasoning as with Lemma 2.10 of \cite{avci2026nehari}, so it's omitted.
\end{proof}
We let $\mathcal{L}(u):=\int_{\Omega}\left(\frac{|u|^{p(x)}}{p(x)|x|^{p(x)}}+\mu(x)\frac{|u|^{q(x)}}{q(x)|x|^{q(x)}}\right)dx $ in the sequel.

\section{Main Results}

\begin{definition}\label{Def:3.1} A function $u\in W_0^{1,\mathcal{H}}(\Omega)$ is called a weak solution to problem (\ref{e1.1}) if it holds
\begin{align}\label{e3.2}
&\int_{\Omega}(|\nabla u|^{p(x)-2}\nabla u+\mu(x)|\nabla u|^{q(x)-2}\nabla u)\cdot \nabla \varphi dx\nonumber\\
&+\int_{\Omega}\left(|x|^{-p(x)}|u|^{p(x)-1}+|x|^{-q(x)}\mu(x)|u|^{q(x)-1}\right) \varphi dx\nonumber\\
&=\int_{\Omega}f(x,u,\nabla u)\varphi dx,\,\,\ \forall \varphi\in W_0^{1,\mathcal{H}}(\Omega).
\end{align}
\end{definition}
Let us define the functionals $\mathcal{T},\mathcal{F}:W_0^{1,\mathcal{H}}(\Omega)\rightarrow \mathbb{R}$ as
$$
\mathcal{T}(u):=\varrho_{\mathcal{H}}(u)+\mathcal{L}(u),
$$
and
$$
\mathcal{F}(u):=\int_{\Omega}f(x,u,\nabla u)dx.
$$
Therefore, we can define the operator $I:W_0^{1,\mathcal{H}}(\Omega)\rightarrow W_0^{1,\mathcal{H}}(\Omega)^{*}$ by
\begin{equation} \label{e2.4n}
\langle I(u),\varphi\rangle:=\langle \mathcal{T}^{\prime}(u),\varphi\rangle-\langle \mathcal{F}^{\prime}(u),\varphi\rangle,\,\, \mbox{ for all } u,\varphi \in W_0^{1,\mathcal{H}}(\Omega).
\end{equation}
Within the framework of monotone operator theory \cite{zeidler2013nonlinear}, and in view of the definitions in (\ref{e3.2}) and (\ref{e2.4n}), the fact that $u\in W_0^{1,\mathcal{H}}(\Omega)$ solves problem (\ref{e1.1}) for every (test function) $\varphi \in W_0^{1,\mathcal{H}}(\Omega)$ can be established by showing that $u$ satisfies the associated operator equation
\begin{equation} \label{e4.2a}
\mathcal{T}^{\prime}u=\mathcal{F}^{\prime}u,\,\,\,u \in W_0^{1,\mathcal{H}}(\Omega).
\end{equation}
Regarding the nonlinearity $f$, we impose the following hypotheses:
\begin{itemize}
\item[$(f_{1})$] $f:\Omega \times \mathbb{R} \times \mathbb{R}^{N} \to \mathbb{R}$ is a Carathéodory function such that $ f(\cdot,0,0)\neq 0$ satisfying
\begin{align*}
&x \longmapsto f(x,t,\eta) \text { is measurable for all } (t,\eta) \in \mathbb{R} \times \mathbb{R}^{N},
\end{align*}
\begin{align*}
&(t,\eta) \longmapsto f(x,t,\eta) \text { is continuous for a.a. } x \in \Omega.
\end{align*}
\item[$(f_{2})$] There exists $g \in L^{r^\prime(x)}(\Omega)$, $r^\prime(x):= \frac{r(x)}{r(x)-1}$, and constants $a_{1},a_{2}>0$ satisfying
\begin{equation*}
|f(x,t,\eta)| \leq g(x)+a_{1}|t|^{r(x)-1} +a_{2}|\eta|^{\frac{p(x)}{r^\prime(x)}},
\end{equation*}
for all $(t,\eta) \in \mathbb{R} \times \mathbb{R}^{N}$ and for a.a. $x\in \Omega$, where $r\in C_+(\overline{\Omega}) $ and $r(x)<p^*(x)$.
\item[$(f_{3})$] There exist $h \in L^{p^\prime(x)}(\Omega)$, $\beta_1 \in L^{\frac {q(x)}{q(x)-\alpha^-}}(\Omega)$ and $\beta_2 \in L^{\frac {q(x)}{q(x)-\alpha(x)}}(\Omega)$ satisfying
\begin{equation*}
\limsup_{|\eta| \to +\infty} \frac{f(x,t,\eta)}{h(x)+\beta_1(x)|t|^{\alpha^--1}+\beta_2(x)|\eta|^{\alpha(x)-1}} \leq \lambda,
\end{equation*}
for all $(t,\eta) \in \mathbb{R} \times \mathbb{R}^{N}$ and for a.a. $x\in \Omega$, where $\alpha \in C_+(\overline{\Omega})$ such that $\alpha^+<p^-$, and $\lambda>0$ is a parameter.
\end{itemize}

\begin{example}
The function $f:\Omega\times\mathbb{R}\times\mathbb{R}^N\to\mathbb{R}$ defined below satisfies hypotheses $(f_1)-(f_3)$:\\
\begin{equation}\label{eq:f-def}
f(x,t,\eta)=c(x)+a(x)|t|+ b(x)|\eta|,
\end{equation}
where $ a,b,c\in C(\overline\Omega)$ such that $a(x)>0,\;b(x)>0,\;c(x)\ge1$, $\forall x\in\overline\Omega$.\\
Hypothesis ($f_1$): It is easy to see that the map $x\mapsto f(x,t,\eta)$ is measurable; and for each fixed $x$, $(t,\eta)\mapsto c(x)+a(x)|t|+b(x)|\eta|$ is continuous; and $f(x,0,0)=c(x)\ge1\neq0$. Thus, hypothesis ($f_1$) is satisfied.\\
Hypothesis ($f_2$): let $r^{-}=\underset{x\in \overline{\Omega }}{\min }r(x)>2$. We use Young’s inequality in the form
\begin{equation}\label{eq:f-deg}
s \leq \varepsilon s^{\gamma-1}+C(\varepsilon), \quad s\ge0,\ \gamma>1.
\end{equation}
Notice that since it must hold $\frac{p(x)}{r'(x)}>1$ (to apply Young’s inequality), and $r^{-}>2$, it reads $p^->2$.\\
Now, we apply (\ref{eq:f-deg}) twice with
$$
(s,\gamma,\varepsilon)=(|t|,r(x),\varepsilon_1),
  \quad
(s,\gamma,\varepsilon)=\left(|\eta|,\frac{p(x)}{r'(x)},\varepsilon_2\right).
$$
Thus, for any fixed $\varepsilon_1,\varepsilon_2>0$ there are constants $C_1,C_2>0$ such that
\begin{align*}
|t| &\leq \varepsilon_1|t|^{\,r(x)-1} + C_1,\\
|\eta| &\leq \varepsilon_2|\eta|^{\frac{p(x)}{r'(x)}} + C_2.
\end{align*}
Substituting these into (\ref{eq:f-def}) gives
\begin{equation*}
|f(x,t,\eta)| \leq c(x)+ a(x)\bigl[\varepsilon_1\,|t|^{\,r(x)-1}+C_1\bigr]+ b(x)\bigl[\varepsilon_2\,|\eta|^{\tfrac{p(x)}{r'(x)}}+C_2\bigr],
\end{equation*}
or
\begin{equation*}
|f(x,t,\eta)| \leq \underbrace{c(x)+a(x)C_1+b(x)C_2}_{:=g(x)}+\bigl(a(x)\varepsilon_1\bigr)\,|t|^{\,r(x)-1}+\bigl(b(x)\varepsilon_2\bigr)\,|\eta|^{\tfrac{p(x)}{r'(x)}}.
\end{equation*}
Since $c,a,b$ are continuous and bounded on $\overline\Omega\), \(g\in L^{r'(\cdot)}(\Omega)$. Finally, if we set
$$
  a_1:=\max_{x\in\overline\Omega}\bigl(a(x)\,\varepsilon_1\bigr),
  \quad
  a_2:=\max_{x\in\overline\Omega}\bigl(b(x)\,\varepsilon_2\bigr),
$$
then $(f_2)$ follows.\\
Hypothesis ($f_3$): Let $h(x)=c(x)$, $\beta_1(x)=a(x)+1$, and $\beta_2(x)=b(x)+1$. Since \(p^->2\), we can choose any continuous function $\alpha(x)\in(2,p(x))$. Then
$\alpha(x)-1>1$, so as $|\eta| \to +\infty$, it follows
\begin{equation*}
\frac{c(x)+ a(x)|t|+ b(x)|\eta|}{c(x)+(a(x)+1)|t|^{\alpha^--1}+(b(x)+1)\lvert \eta\rvert^{\alpha(x)-1}}\leq 1.
\end{equation*}
Thus
\[
\limsup_{|\eta|\to +\infty}
\frac{\lvert f(x,t,\eta)\rvert}
     {\,h(x)+\beta_1(x)\lvert t\rvert^{\alpha^--1}
             +\beta_2(x)\lvert \eta\rvert^{\alpha(x)-1}}
\;\le\;1=\lambda.
\]
Note that since $a,b,c$ are continuous and bounded, we have $h\in L^{p'(x)}(\Omega)$, $\beta_1 \in L^{\frac {q(x)}{q(x)-\alpha^-}}(\Omega)$ and $\beta_2 \in L^{\frac {q(x)}{q(x)-\alpha(x)}}(\Omega)$.
\newpage
\end{example}

\begin{lemma}\label{Lem:3.1} For the operators $\mathcal{T}^{\prime}$  and $\mathcal{F}^{\prime}$, we have:
\begin{itemize}
\item[$(i)$] $\mathcal{T}^{\prime}$ is a strictly monotone operator.
\item[$(ii)$] $\mathcal{T}^{\prime}$ is of type $(S_{+})$, that is, if $u_{n}\rightharpoonup u$ in $W_0^{1,\mathcal{H}}(\Omega)$ and
\newline $\limsup_{n\rightarrow\infty}\langle \mathcal{T}^{\prime}(u_{n})-\mathcal{T}^{\prime}(u),u_{n}-u\rangle\leq 0$ then $u_{n}\rightarrow u$ in $W_0^{1,\mathcal{H}}(\Omega)$.
\item[$(iii)$] $\mathcal{T}^{\prime}$ is a homeomorphism.
\item[$(iv)$] $\mathcal{F}^{\prime}$ is compact.
\end{itemize}
\end{lemma}

\begin{proof} First, recall that
$$
\mathcal{T}(u)=\varrho_{\mathcal{H}}(u)+\mathcal{L}(u)\quad \text{and} \quad \mathcal{F}(u)=\int_{\Omega}f(x,u,\nabla u)dx.
$$
Since strict monotonicity, and being of type $(S_{+})$ properties for $\varrho_{\mathcal{H}}(\cdot)$ has been given in \cite[Theorem 3.3]{crespo2022new}, we skip those parts in the proofs of $(i)-(ii)$ below, and continue only with the operator $\mathcal{L}:W_0^{1,\mathcal{H}}(\Omega)\rightarrow \mathbb{R}$  whose derivative is
$$
\langle\mathcal{L}^{\prime}_{\mathcal{H}}(u),\varphi\rangle=\int_{\Omega}\left(|x|^{-p(x)}|u|^{p(x)-1}+|x|^{-q(x)}\mu(x)|u|^{q(x)-1}\right) \varphi dx.
$$
Indeed, it is quite straightforward to show that $\mathcal{L}$ is of class $C^{1}(W_0^{1,\mathcal{H}}(\Omega),\mathbb{R})$. However, since $\mathcal{L}$ involves singularities, we provide a concise proof for the sake of completeness.\\
To this end, we first show the continuity of $\mathcal{L}$. By the mean value theorem there are  $0< \varepsilon_1, \varepsilon_2< 1$ such that

\begin{align}\label{e3.11d}
\langle \mathcal{L}^{\prime}(u),\varphi \rangle&=\lim_{t \to 0}\int_{\Omega}\frac{1}{t}\left(\frac{|u+t\varphi|^{p(x)}-|u|^{p(x)}}{p(x)|x|^{p(x)}}+\frac{\mu(x)|u+t\varphi|^{q(x)}-|u|^{q(x)}}{q(x)|x|^{q(x)}}\right)dx \nonumber \\
&=\lim_{t \to 0}\int_{\Omega}\frac{1}{t}\left(\frac{d}{d\epsilon}\frac{|u+\epsilon t\varphi|^{p(x)}}{p(x)|x|^{\alpha p(x)}}\bigg|_{\epsilon=\varepsilon_{1}}+\frac{d}{d\epsilon}\frac{\mu(x)|u+\epsilon t\varphi|^{q(x)}}{q(x)|x|^{\alpha q(x)}}\bigg|_{\epsilon=\varepsilon_{2}}\right)dx  \nonumber \\
& =\lim_{t \to 0}\int_{\Omega}\left(|x|^{- p(x)}|u+\varepsilon_1 t\varphi|^{p(x)-2}(u+\varepsilon_1 t\varphi) \right. \nonumber \\
&+ \left. |x|^{- q(x)}\mu(x)|u+\varepsilon_2 t\varphi|^{q(x)-2}(u+\varepsilon_2 t\varphi)\right) \varphi  dx, \nonumber \\
\end{align}
for all $u, \varphi \in W_0^{1,\mathcal{H}}(\Omega)$, $\epsilon \in \mathbb{R}$. Using Young's inequality, we obtain
\begin{align}\label{e3.11f}
  \bigg||x|^{-p(x)}|u+\varepsilon_1  t\varphi|^{p(x)-2}(u+\varepsilon_1 t\varphi) \varphi \bigg| & \leq c_0 |x|^{-p(x)}(|u|^{p(x)}+|\varphi|^{p(x)}),
\end{align}
and
\begin{align}\label{e3.11fg}
  \bigg||x|^{-q(x)}\mu(x)|u+\varepsilon_2 t\varphi|^{q(x)-2}(u+\varepsilon_2 t\varphi) \varphi\bigg| & \leq c_0  |x|^{-q(x)}(\mu(x)|u|^{q(x)}+\mu(x)|\varphi|^{q(x)}),
\end{align}
where $c_0 :=\frac{2^{q^+}(q^+-1)}{p^-}$.\\ Employing (\ref{e3.11f}), (\ref{e3.11fg}) and Lemma \ref{Lem:2.1} all together gives
\begin{align}\label{e3.11fh}
&\bigg||x|^{-p(x)}|u+\varepsilon_1  t\varphi|^{p(x)-2}(u+\varepsilon_1 t\varphi) \varphi \bigg|+\bigg||x|^{-q(x)}\mu(x)|u+\varepsilon_2 t\varphi|^{q(x)-2}(u+\varepsilon_2 t\varphi) \varphi\bigg| \nonumber \\
& \leq c_0 (|x|^{-p(x)}|u|^{p(x)}+|x|^{-q(x)}\mu(x)|\varphi|^{q(x)})+c_0 (|x|^{-p(x)}|\varphi|^{p(x)}+|x|^{-q(x)}\mu(x)|\varphi|^{q(x)})\nonumber \\
& \leq c_0  q^+ C_{N}(p,q)\left(\|u\|_{1,\mathcal{H},0}^{\kappa}+\|\varphi\|_{1,\mathcal{H},0}^{\kappa}\right).
\end{align}
Then, using the Lebesgue dominated convergence theorem provides
\begin{align}\label{e3.11g}
\langle \mathcal{L}^{\prime}(u),\varphi \rangle & =\int_{\Omega}\lim_{t \to 0}\left(|x|^{-p(x)}|u+\varepsilon_1 t\varphi|^{p(x)-2}(u+\varepsilon_1 t\varphi) \right. \nonumber \\
&+ \left. |x|^{-q(x)}\mu(x)|u+\varepsilon_2 t\varphi|^{q(x)-2}(u+\varepsilon_2 t\varphi)\right) \varphi  dx, \nonumber \\
&=\int_{\Omega}\left(|x|^{-p(x)}|u|^{p(x)-1}+\mu(x)|x|^{-q(x)}|u|^{q(x)-1}\right) \varphi dx.
\end{align}
Due to the fact that the right-hand side of (\ref{e3.11g}) is a continuous linear functional (of $\varphi$) on $W_0^{1,\mathcal{H}}(\Omega)$, it is the G\^{a}teaux derivative of $\mathcal{L}$ .\\
Next,using the Young inequality and Lemma \ref{Lem:2.1} together provides
\begin{align}\label{e3.22g}
|\langle \mathcal{L}^{\prime}(u),\varphi \rangle| & \leq \int_{\Omega}\left(|x|^{-p(x)}|u|^{p(x)-1} |\varphi|+|x|^{- q(x)}\mu(x)|u|^{q(x)-1} |\varphi|\right) dx \nonumber \\
& \leq (q^+-1) \int_{\Omega}\left(\frac{|u|^{p(x)}}{p(x)|x|^{ p(x)}}+\frac{|\varphi|^{p(x)}}{p(x)|x|^{ p(x)}}\right)
+\left(\frac{\mu(x)|u|^{q(x)}}{q(x)|x|^{ q(x)}}+\frac{\mu(x)|\varphi|^{q(x)}}{q(x)|x|^{ q(x)}}\right)dx \nonumber \\
& \leq (q^+-1)C_{N}(p^-,q^+)\left(\|u\|_{1,\mathcal{H},0}^{\kappa} + \|\varphi\|_{1,\mathcal{H},0}^{\kappa}\right).
\end{align}
Hence, for all $u \in W_0^{1,\mathcal{H}}(\Omega)$, we have
\begin{align}\label{e3.21gb}
\|\mathcal{L}^{\prime}(u)\|_{W_0^{1,\mathcal{H}}(\Omega)^*}=\sup_{\|\varphi\|_{1,\mathcal{H},0}\leq 1}|\langle \mathcal{L}^{\prime}(u),\varphi \rangle| \leq (q^+-1)C_{N}(p^-,q^+)(\|u\|_{1,\mathcal{H},0}^{\kappa}+1),
\end{align}
that is, $\mathcal{L}^{\prime}$ is bounded. Therefore, $\mathcal{L}$ is G\^{a}teaux differentiable whose derivative is given by the formula (\ref{e3.11g}).
Now, we proceed with the continuity of $\mathcal{G}^{\prime}$. Let $(u_n) \subset W_0^{1,\mathcal{H}}(\Omega)$ such that $u_n \to u$ in $W_0^{1,\mathcal{H}}(\Omega)$. Applying the inequality (Proposition 17.2, \cite{chipot2009elliptic})
\begin{equation}\label{e3.21gbr}
\left|\left\vert \xi\right\vert ^{s-2}\xi-\left\vert \psi\right\vert^{s-2}\psi \right| \leq C_{s}\left\vert\xi-\psi\right\vert\{|\xi|+|\psi|\}^{s-2}, \quad \xi,\psi\in \mathbb{R}^{N},\,\, s> 1,
\end{equation}
gives
\begin{align}\label{e3.21gbd}
\left|\langle\mathcal{L}^{\prime}(u_{n})-\mathcal{G}^{\prime}(u),\varphi\rangle\right|& \leq \int_{\Omega}|x|^{-p(x)}\left||u_{n}|^{p(x)-2}u_{n}-|u|^{p(x)-2}u \right||\varphi| dx \nonumber \\
&+ \int_{\Omega}|x|^{- q(x)}\mu(x)\left||u_{n}|^{q(x)-2}u_{n}-|u|^{q(x)-2}u \right||\varphi| dx \nonumber \\
& \leq C_{p}\int_{\Omega}|x|^{-p(x)}\{|u_{n}|+|u|\}^{p(x)-2}|u_{n}-u| |\varphi|dx \nonumber \\
&+ C_{q}\int_{\Omega}|x|^{- q(x)}\mu(x)\{|u_{n}|+|u|\}^{q(x)-2}|u_{n}-u||\varphi| dx.
\end{align}
Lastly, we proceed with the continuity of $\mathcal{L}^{\prime}$. To this end, for a sequence $(u_n) \subset W_0^{1,\mathcal{H}}(\Omega)$ such that $u_n \to u$ in $W_0^{1,\mathcal{H}}(\Omega)$, using the fact that $\rho \in C^{1}(W_0^{1,\mathcal{H}}(\Omega),\mathbb{R})$, it reads
\begin{equation}\label{e3.11k}
\limsup_{n \to \infty} \|\mathcal{L}^{\prime}(u_n)-\mathcal{L}^{\prime}(u)\|_{W_0^{1,\mathcal{H}}(\Omega)^*} \to 0.
\end{equation}
Therefore, $\mathcal{L}$ is of class $C^{1}(W_0^{1,\mathcal{H}}(\Omega),\mathbb{R})$.\\
$\mathbf{(i)}$ For $u, v \in W_0^{1,\mathcal{H}}(\Omega)$ with $u\neq v$, we have
\begin{align}\label{e3.11zz}
  \langle \mathcal{L}^{\prime}(u)-\mathcal{L}^{\prime}(v), u- v \rangle &= \int_{\Omega}|x|^{-p(x)}\left(|u|^{p(x)-2}u-|v|^{p(x)-2}\upsilon \right)(u-v) dx \nonumber \\
  &+ \int_{\Omega}|x|^{-q(x)}\mu(x)\left(|u|^{q(x)-2}u-|v|^{q(x)-2}\upsilon \right)(u-v) dx \nonumber \\
  &\geq 2^{-q^{+}}M^{-\tau}\left(\int_{\Omega}|u-v|^{p(x)}dx +\int_{\Omega}\mu(x)|u-v|^{q(x)}dx\right) > 0,
\end{align}
where we apply the well-known inequality
\begin{equation*}
\left( \left\vert \xi\right\vert ^{s-2}\xi-\left\vert \psi\right\vert^{s-2}\psi \right)\cdot\left(\xi-\psi\right) \geq 2^{-s}\left\vert\xi-\psi\right\vert^{s}; \quad \xi,\psi\in \mathbb{R}^{N},\,\, s> 1.
\end{equation*}
$\mathbf{(ii)}$ Let $(u_{n})\subset W_0^{1,\mathcal{H}}(\Omega)$ be a sequence satisfying
\begin{equation}\label{e3.2a}
u_{n} \rightharpoonup u \in W_0^{1,\mathcal{H}}(\Omega),
\end{equation}
\begin{equation}\label{e3.2b}
\limsup_{n\rightarrow\infty}\langle \mathcal{L}^{\prime}(u_{n}),u_{n}-u\rangle\leq 0.
\end{equation}
We need to show that $u_{n} \to u \in W_0^{1,\mathcal{H}}(\Omega)$. \\
Note that since $u_{n} \rightharpoonup u \in W_0^{1,\mathcal{H}}(\Omega)$, we have

\begin{itemize}
\item  $u_{n}\rightarrow u$ in $L^{p(x)}(\Omega)$,
\item  $u_{n}\rightarrow u$ in $L^{q(x)}(\Omega)$,
\item  $u_{n}(x) \rightarrow u(x) $ a.e. in $\Omega$.
\end{itemize}
Thus, considering these facts and using Lemma \ref{Lem:2.1} and Propositions \ref{Prop:2.2}-\ref{Prop:2.2a}, we obtain that $u_{n} \to u \in W_0^{1,\mathcal{H}}(\Omega)$.\\
$\mathbf{(iii)}$ We need to show $(\mathcal{T}^{\prime })^{-1}:W_0^{1,\mathcal{H}}(\Omega)^{*}\rightarrow W_0^{1,\mathcal{H}}(\Omega)$ is continuous. In doing so, we shall first show that $\mathcal{T}^{\prime}$ is coercive. We assume that $\|u\|_{1,\mathcal{H},0}>1$. Thus, using Lemma \ref{Lem:2.1} and Proposition \ref{Prop:2.2a}, we obtain
\begin{align}\label{e3.3}
\langle \mathcal{T}^{\prime}(u),u\rangle&=\int_{\Omega}\left(|\nabla u|^{p(x)}+\mu(x)|\nabla u|^{q(x)}\right) dx \nonumber\\
&+\int_{\Omega}\left(|x|^{-p(x)}|u|^{p(x)}+|x|^{-q(x)}\mu(x)|u|^{q(x)}\right)dx\nonumber\\
&\geq \|u\|^{p^-}_{1,\mathcal{H},0}+M^{-\tau}\|u\|^{p^-}_{\mathcal{H}},
\end{align}
which means that $\frac{\langle \mathcal{T}^{\prime}(u),u\rangle }{\|u\|_{1,\mathcal{H},0}} \to \infty$ as $\|u\|_{1,\mathcal{H},0}\to \infty$.\\
Thus, since $\mathcal{T}^{\prime}$ is also strictly monotone, $\mathcal{T}^{\prime}$ is an injection. By Minty-Browder theorem (see \cite{zeidler2013nonlinear}), however, these two properties together mean that $\mathcal{T}^{\prime }$ is a surjection. As a consequence, $\mathcal{T}^{\prime }$ has an inverse mapping $(\mathcal{T}^{\prime })^{-1}:W_0^{1,\mathcal{H}}(\Omega)^{*}\rightarrow W_0^{1,\mathcal{H}}(\Omega)$. As for the continuity of $(\mathcal{T}^{\prime })^{-1}$, let $(u_{n}^{\ast }),u^{\ast }\in W_0^{1,\mathcal{H}}(\Omega)^{*}$ with $
u_{n}^{\ast }\rightarrow u^{\ast }$, and let $(\mathcal{T}^{\prime })^{-1}(u_{n}^{\ast })=u_{n},(\mathcal{T}^{\prime})^{-1}(u^{\ast
})=u$. Then, $\mathcal{T}^{\prime} (u_{n})=u_{n}^{\ast }$ and $\mathcal{T}^{\prime}(u)=u^{\ast }$ which
means, by the coercivity of $\mathcal{T}^{\prime}$, that $( u_{n}) $ is bounded in $W_0^{1,\mathcal{H}}(\Omega)$. Therefore, there exist $\hat{u}\in W_0^{1,\mathcal{H}}(\Omega)$ and a subsequence, not relabelled, $(u_{n})\subset W_0^{1,\mathcal{H}}(\Omega)$ such that $u_{n}\rightharpoonup \hat{u}$ in $W_0^{1,\mathcal{H}}(\Omega)$. However, by the uniqueness of the weak limit, we need to have $u_{n}\rightharpoonup u$ in $W_0^{1,\mathcal{H}}(\Omega)$. Furthermore, since $u_{n}^{\ast}\rightarrow u^{\ast}$ in $W_0^{1,\mathcal{H}}(\Omega)^{*}$, it reads
\begin{equation}\label{e3.4}
\lim_{n\rightarrow\infty}\langle u_{n}^{\ast }-u^{\ast },u_{n}-u\rangle =\lim_{n\rightarrow\infty}\langle \mathcal{T}^{\prime}(u_{n})-\mathcal{T}^{\prime}(u),u_{n}-u\rangle= 0.
\end{equation}
Due to the fact that $\mathcal{T}^{\prime}$ is of type $(S_{+})$, we have $u_{n} \rightarrow u$ in $W_0^{1,\mathcal{H}}(\Omega)$. In conclusion, $(\mathcal{T}^{\prime })^{-1}:W_0^{1,\mathcal{H}}(\Omega)^{*}\rightarrow W_0^{1,\mathcal{H}}(\Omega)$ is continuous.\\
$\mathbf{(iv)}$ Define the operator $\mathcal{A}_f:W_0^{1,\mathcal{H}}(\Omega)\rightarrow L^{r^{\prime}(x)}(\Omega)$ by
\begin{equation} \label{e3.5}
\mathcal{A}_f(u)=f(x,u,\nabla u).
\end{equation}
$\mathcal{A}_f$ is $L^{r^{\prime}(x)}(\Omega)$-norm bounded. Indeed, let $\|u\|_{1,\mathcal{H},0}\leq1$. Then using $(f_2)$ and the embeddings the embeddings $L^{\mathcal{H}}(\Omega)\hookrightarrow L^{r(x)}(\Omega)$, and $W_0^{1,\mathcal{H}}(\Omega)\hookrightarrow L^{\mathcal{H}}(\Omega)$, it reads
\begin{align}\label{e3.6}
\int_{\Omega} |\mathcal{A}_f(u)|^{r^\prime(x)} dx&=\int_{\Omega} |f(x,u,\nabla u)|^{r^\prime(x)} dx\nonumber\\
&\leq \int_{\Omega}  |g(x)+a_{1}|u|^{r(x)-1} +a_{2}|\nabla u|^{p(x)\frac{r(x)-1}{r(x)}}|^{r^\prime(x)}dx  \nonumber\\
&\leq c\int_{\Omega} \left(|g(x)|^{r^\prime(x)}+|u|^{r(x)}+|\nabla u|^{p(x)}\right)dx \nonumber\\
&\leq c\int_{\Omega} \left(|g(x)|^{r^\prime(x)}+|u|^{r(x)}+(|\nabla u|^{p(x)}+\mu(x)|\nabla u|^{q(x)})\right)dx \nonumber\\
&\leq c\|u\|_{1,\mathcal{H},0}.
\end{align}
Let $u_n \to u$ in $W_0^{1,\mathcal{H}}(\Omega)$.  Then, by Proposition \ref{Prop:2.7a}, $\nabla u_n \to \nabla u$ in $L^{h(x)}(\Omega)^{N}$ for all $h\in C(\overline{\Omega})$ with $1\leq h(x)< p^*(x)$, $x\in \overline{\Omega}$. This makes sure the existence of a subsequence $(u_n)$, not relabelled, and the functions $\omega_1(x)$ in $L^{h(x)}(\Omega)$ and $\omega_2(x)$ in $L^{h(x)}(\Omega)^{N}$ satisfying
\begin{itemize}
\item $u_n(x) \to u(x)$ and $\nabla u_n(x) \to \nabla u(x)$ a.e. in $\Omega$,
\item $|u_n(x)|\leq \omega_1(x)$ and $|\nabla u_n(x)|\leq |\omega_2(x)|$ a.e. in $\Omega$ and for all $n$.
\end{itemize}
By the assumption $(f_1)$, $f$ is continuous in the last two arguments, thus
\begin{equation} \label{e3.7}
f(x,u_n(x),\nabla u_n(x))\to f(x,u(x),\nabla u(x)) \text{  a.e. in } \Omega \text{ as } n\to \infty.
\end{equation}
Additionally, using the above information and $(f_2)$ gives
\begin{align} \label{e3.8}
|f(x,u_n(x),\nabla u_n(x))|& \leq g(x)+a_{1}|\omega_1(x)|^{r(x)-1}+a_{2}|\omega_2(x)|^{p(x)\frac{r(x)-1}{r(x)}}.
\end{align}
Note that we assume that $|\omega_1(x)|, |\omega_2(x)|>1$ a.e. in $\Omega$, otherwise the right-hand side of (\ref{e3.8}) is already integrable over $\Omega$. Thus,
\begin{equation}\label{e3.9}
\int_{\Omega}a_{1}|\omega_1|^{r(x)-1}dx \leq  a_{1} \int_{\Omega}|\omega_1|^{r(x)}dx \leq c_1 |\omega_1|^{^{r^+}}_{r(\cdot)},
\end{equation}
and similarly
\begin{equation}\label{e3.10}
\int_{\Omega}a_{2}|\omega_2|^{p(x)\frac{r(x)-1}{r(x)}}dx \leq a_{2} \int_{\Omega}|\omega_2|^{p(x)}dx \leq c_2 |\omega_2|^{p^+}_{p(\cdot)}.
\end{equation}
Thus, the right-hand side of (\ref{e3.8}) is integrable. Therefore, by (\ref{e3.7}) and the Lebesgue dominated convergence theorem (see, e.g, \cite{royden2010real}), we obtain
\begin{equation}\label{e3.11}
f(x,u_{n},\nabla u_{n}) \rightarrow f(x,u,\nabla u)\,\, \text {in}\,\, L^{1}(\Omega),
\end{equation}
and hence
\begin{align}\label{e3.12}
\lim_{n \to \infty}\int_{\Omega} |\mathcal{A}_f(u_n)-\mathcal{A}_f(u)|^{r^\prime(x)} dx=0,
\end{align}
which, by Proposition \ref{Prop:2.2}, implies  that $\mathcal{A}_f$ is continuous in $L^{r^{\prime}(x)}(\Omega)$.\\
Next, consider the compact imbedding operator $i: W_0^{1,\mathcal{H}}(\Omega) \to L^{r(x)}(\Omega)$. Then the adjoint operator $i^*$ of $i$, given by $i^*:L^{r^{\prime}(x)}(\Omega) \to W_0^{1,\mathcal{H}}(\Omega)^*$  is also compact. Therefore, if we set $\mathcal{F^{\prime}}= i^*\circ \mathcal{A}_f$, then $\mathcal{F^{\prime}}:W_0^{1,\mathcal{H}}(\Omega)\rightarrow W_0^{1,\mathcal{H}}(\Omega)^{*}$ is compact.
\end{proof}

We shall use the following topological existence result, introduced in \cite[Theorem 4.1]{avci2019topological}, to obtain the main result of the present paper.
\begin{theorem}\label{Thm:3.2}
Suppose that $(f_1)-(f_3)$ hold. Additionally, assume that the following conditions are fulfilled:
\begin{itemize}
\item [$(i)$] $\mathcal{T}^{\prime}$ is a homeomorphism
\item [$(ii)$] $\mathcal{F}^{\prime}$ is compact
\item [$(iii)$] The mapping $\mathcal{T}^{\prime}-\mathcal{F}^{\prime}$ is coercive
\end{itemize}
Then operator equation (\ref{e4.2a}) has a nontrivial solution in $W_0^{1,\mathcal{H}}(\Omega)$, which in turn becomes a solution to problem (\ref{e1.1}).
\end{theorem}
\begin{proof}
Due to the results of Lemma \ref{Lem:3.1}, it is sufficient to prove only part $\mathbf{(iii)}$.\\
Since we look for the coercivity, we may assume that $|\nabla u|>c$ for some constant $c\geq 1$, and hence $\|u\|_{1,\mathcal{H},0}>1$. Indeed, if we recall that $\|u\|_{1,\mathcal{H},0}=\||\nabla u|\|_{\mathcal{H}}$, and take into account the monotonicity of convex modular $\rho_\mathcal{H}(\cdot)$ and Proposition \ref{Prop:2.2a}, we can conclude that $\|u\|_{1,\mathcal{H},0}\geq (\rho_\mathcal{H}(\nabla u))^{1/q^+}>1$.
Therefore, from $(f_3)$, we have
\begin{equation} \label{e3.13}
|f(x,u,\nabla u)| \leq \lambda \left(|h|+|\beta_1||u|^{\alpha^--1}+|\beta_2||\nabla u|^{\alpha(x)-1}\right).
\end{equation}
Using H\"{o}lder inequality \cite[Proposition 2.3]{avci2019topological}, Proposition \ref{Prop:2.2bb} and the necessary embeddings, we get
\begin{align}\label{e3.13a}
|\langle \mathcal{F}^{\prime}u,u\rangle| &\leq \int_{\Omega}\lambda \left(|h||u|+|\beta_1||u|^{\alpha^--1}|u|+|\beta_2||\nabla u|^{\alpha(x)-1}|u|\right)dx\nonumber\\
&\leq \lambda \left(|h|_{p^{\prime}}|u|_{p(x)}+|\beta_1|_{\frac{q(x)}{q(x)-\alpha^-}}||u|^{\alpha^--1}|_{\frac{q(x)}{\alpha^--1}}|u|_{q(x)}\right. \nonumber\\
&\left. +|\beta_2|_{\frac{q(x)}{q(x)-\alpha(x)}}||\nabla u|^{\alpha(x)-1}|_{\frac{q(x)}{\alpha(x)-1}}|u|_{q(x)} \right)\nonumber\\
& \leq \lambda \left(c_1\|u\|^{\alpha^+}_{1,\mathcal{H},0}+c_2\|u\|^{\alpha^-}_{1,\mathcal{H},0}+c_3\|u\|_{1,\mathcal{H},0} \right),
\end{align}
and hence
\begin{equation} \label{e3.13b}
\frac{|\langle \mathcal{F}^{\prime}u,u\rangle|}{\|u\|_{1,\mathcal{H},0}} \leq \lambda \left(c_1\|u\|^{\alpha^+-1}_{1,\mathcal{H},0}+c_2\|u\|^{\alpha^--1}_{1,\mathcal{H},0}+c_3 \right).
\end{equation}
Similarly, using Lemma \ref{Lem:2.1} and Proposition \ref{Prop:2.2a}, it reads
\begin{align}\label{e3.13c}
|\langle \mathcal{T}^{\prime}u,u\rangle|&=\bigg|\int_{\Omega}\left(|\nabla u|^{p(x)}+\mu(x)|\nabla u|^{q(x)}\right) dx+\int_{\Omega}\left(|x|^{-p(x)}|u|^{p(x)}+|x|^{-q(x)}\mu(x)|u|^{q(x)}\right)dx\bigg| \nonumber\\
&\geq \|u\|_{1,\mathcal{H},0}^{p^-}+M^{-\tau}\|u\|^{p^-}_{\mathcal{H}},
\end{align}
and hence
\begin{align}\label{e3.13d}
\frac{|\langle \mathcal{T}^{\prime}u,u\rangle|}{\|u\|_{1,\mathcal{H},0}} \geq \|u\|_{1,\mathcal{H},0}^{p^--1}+M^{-\kappa}\|u\|^{p^--1}_{\mathcal{H}}.
\end{align}
Using (\ref{e2.4n}), we can let $\langle I(u),u\rangle=\langle \mathcal{T}^{\prime}(u),u\rangle-\langle \mathcal{F}^{\prime}(u),u\rangle$. Then from (\ref{e3.13b}) and (\ref{e3.13d}), we have
\begin{align}\label{e3.14}
\frac{|\langle \mathcal{I}(u),u\rangle|}{\|u\|_{1,\mathcal{H},0}}& \geq \frac{|\langle \mathcal{T}^{\prime}u,u\rangle|}{\|u\|_{1,\mathcal{H},0}}-\frac{|\langle \mathcal{F}^{\prime}u,u\rangle|}{\|u\|_{1,\mathcal{H},0}}\nonumber\\
& \geq \|u\|_{1,\mathcal{H},0}^{p^--1}+M^{-\tau}\|u\|^{p^--1}_{\mathcal{H}}-\lambda \left(c_1\|u\|^{\alpha^+-1}_{1,\mathcal{H},0}+c_2\|u\|^{\alpha^--1}_{1,\mathcal{H},0}+c_3 \right),
\end{align}
which implies
\begin{align}\label{e3.15}
\frac{|\langle \mathcal{I}(u),u\rangle|}{\|u\|_{1,\mathcal{H},0}}\to \infty \text{ as } \|u\|_{1,\mathcal{H},0} \to \infty.
\end{align}
Hence the mapping $\mathcal{T}^{\prime}-\mathcal{F}^{\prime}$ is coercive.\\

The coercivity allow us to work outside of a ball. That is, there is a constant $r_{0}>1$ satisfying
\begin{equation}\label{e4.6a}
\|\mathcal{T}^{\prime}u-\mathcal{F}^{\prime}u\|_{W_0^{1,\mathcal{H}}(\Omega)^{*}} > 1,\,\,\, \forall u\in W_0^{1,\mathcal{H}}(\Omega),\,\, \|u\|_{1,\mathcal{H},0}\geq r_{0}.
\end{equation}
On the other hand, since $\mathcal{T}^{\prime}:W_0^{1,\mathcal{H}}(\Omega)\to W_0^{1,\mathcal{H}}(\Omega)^{*}$ is a homeomorphism, (\ref{e4.2a}) can be equivalently written as
\begin{equation}\label{e4.6b}
u=(\mathcal{T}^{\prime})^{-1}(\mathcal{F}^{\prime}u).
\end{equation}
Next, let us define the compact operator $\mathcal{K}:=(\mathcal{T}^{\prime})^{-1}(\mathcal{F}^{\prime}):W_0^{1,\mathcal{H}}(\Omega)\rightarrow W_0^{1,\mathcal{H}}(\Omega)$. Therefore, a solution of (\ref{e4.6b}) will be the solution to the equation
\begin{equation}\label{e4.6c}
u=\mathcal{K}u.
\end{equation}
However, with this setting, we can now use the Leray–Schauder degree and homotopy mapping together to prove the existence of solutions to nonlinear equation (\ref{e4.6c}). To do so, we need to define the set
\begin{equation}\label{e4.6d}
\mathcal{G}=\{u\in W_0^{1,\mathcal{H}}(\Omega):\,\, u=\gamma (\mathcal{T}^{\prime})^{-1}(\mathcal{F}^{\prime}u)\,\,\, \text{for some}\,\,\, \gamma \in [0,1] \}.
\end{equation}
First, we show that $\mathcal{G}$ is bounded. To this end, for any $u\in \mathcal{G}\backslash\{0\}$ with $\|u\|_{1,\mathcal{H},0}>1$, we write
\begin{align}\label{e4.7}
\frac{|\langle \mathcal{T}^{\prime}\frac{u}{\gamma},\frac{u}{\gamma}\rangle|}{\|\frac{u}{\gamma}\|_{1,\mathcal{H},0}}=\|\mathcal{T}^{\prime}\frac{u}{\gamma}\|_{W_0^{1,\mathcal{H}}(\Omega)^{*}}
=\|\mathcal{F}^{\prime}u\|_{W_0^{1,\mathcal{H}}(\Omega)^{*}}=\frac{|\langle \mathcal{F}^{\prime}u,u\rangle|}{\|u\|_{1,\mathcal{H},0}}.
\end{align}
Then using (\ref{e3.13b}) and (\ref{e3.13d}), we have
\begin{align}\label{e4.8}
\gamma^{1-q^+}\left(\|u\|_{1,\mathcal{H},0}^{p^--1}+M^{-\tau}\|u\|^{p^--1}_{\mathcal{H}}\right) \leq \lambda \left(c_1\|u\|^{\alpha^+-1}_{1,\mathcal{H},0}+c_2\|u\|^{\alpha^--1}_{1,\mathcal{H},0}+c_3\right),
\end{align}
which implies that $\mathcal{G}$ is bounded in $W_0^{1,\mathcal{H}}(\Omega)$.\\
Therefore, there is some constant $r_{1}\geq r_{0}$ such that $\mathcal{G}\subseteq B_{r_{1}}(0)$ holds. Thus, the compact operator $\mathcal{K}$ can be defined by $\mathcal{K}:\overline{B_{r_{1}}(0)}\rightarrow W_0^{1,\mathcal{H}}(\Omega)$. As a result of (\ref{e4.6a}), we have $u-\mathcal{K}u\neq 0$ for any $u \in \partial B_{r_{1}}(0)$. Otherwise, we would have $u=\mathcal{K}u=(\mathcal{T}^{\prime})^{-1}(\mathcal{F}^{\prime}u)$ for any $u \in \partial B_{r_{1}}(0)$; but, this would imply $(\mathcal{T}^{\prime}-\mathcal{F}^{\prime})u=0$ which contradicts (\ref{e4.6a}). Therefore, we can associate the Leray-Schauder degree of mapping, a $\mathbb{Z}$-valued function $d(I-\mathcal{K},B_{r_{1}}(0),0)$, to $\mathcal{K}$. \\
Next, let us define the mapping
\begin{equation}\label{e4.9}
H(u,t)=u-t\mathcal{K}u\,\,\, \text{for}\,\,\, u\in \overline{B_{r_{1}}(0)}\,\,\, \text{and}\,\,\, t\in [0,1].
\end{equation}
Then $H(u,t)$ is a continuous mapping on $\overline{B_{r_{1}}(0)}\times[0,1]$ satisfying $H(u,t)\neq 0$ for all $u\in \partial B_{r_{1}}(0)$ and $t\in [0,1]$. Indeed, if we argue by contradiction, and assume that there exist $\hat{u}\in \partial B_{r_{1}}(0)$ and $\hat{t}\in [0,1]$ such that $\hat{u}-\hat{t}\mathcal{K}\hat{u}=0$. Then
\begin{equation}\label{e4.11}
0=\|\hat{u}-\hat{t}\mathcal{K}\hat{u}\|_{1,\mathcal{H},0}\geq \|\hat{u}\|_{1,\mathcal{H},0}-\hat{t}\|\mathcal{K}\hat{u}\|_{1,\mathcal{H},0}\geq (1-\hat{t})r_{1}\geq 0.
\end{equation}
However, (\ref{e4.11}) means that $\|\mathcal{K}\hat{u}\|_{1,\mathcal{H},0} \leq r_{1}$ since $\|\hat{u}\|_{1,\mathcal{H},0} = r_{1}$. Therefore, we conclude that $\hat{t}=1$ which contradicts the fact $u-\mathcal{K}u\neq 0$ for any $u \in \partial B_{r_{1}}(0)$.\\
Therefore,
\begin{equation}\label{e4.12}
H(u,t)\neq 0\,\,\, \text{for}\,\,\, u\in \partial B_{r_{1}}(0)\,\,\, \text{and}\,\,\, t\in [0,1].
\end{equation}
Thus, $H(\cdot,t)$ is a homotopy of the mappings $I=H(\cdot,0)$ and $I-\mathcal{K}=H(\cdot,1)$. Using the homotopy invariance and normalization properties of degree, we have
\begin{equation}\label{e4.13}
d(I-\mathcal{K},B_{r_{1}}(0),0)=d(I,B_{r_{1}}(0),0)=1,
\end{equation}
which implies that $\mathcal{K}$ has a fixed point located in $B_{r_{1}}(0)$. In conclusion, the operator equation (\ref{e4.6c}) (and hence (\ref{e4.2a})) has a nontrivial solution in $W_0^{1,\mathcal{H}}(\Omega)$, which is a nontrivial solution for problem (\ref{e1.1}).
\end{proof}

\section*{Declarations}
\vspace{0.5cm}
\section*{Ethical Approval}
Not applicable.

\section*{Data Availability}
No data is used to conduct this research.

\section*{Funding}
This work was supported by Athabasca University Research Incentive Account [140111 RIA].

\section*{ORCID}
https://orcid.org/0000-0002-6001-627X

\bibliographystyle{tfnlm}
\bibliography{references}

@article{cencelj2018double,
  title={Double phase problems with variable growth},
  author={Cencelj, Matija and R{\u{a}}dulescu, Vicen{\c{t}}iu D and Repov{\v{s}}, Du{\v{s}}an D},
  journal={Nonlinear Analysis},
  volume={177},
  pages={270--287},
  year={2018},
  publisher={Elsevier}
}

@article{bahrouni2019double,
  title={Double phase transonic flow problems with variable growth: nonlinear patterns and stationary waves},
  author={Bahrouni, Anouar and R{\u{a}}dulescu, Vicen{\c{t}}iu D and Repov{\v{s}}, Du{\v{s}}an D},
  journal={Nonlinearity},
  volume={32},
  number={7},
  pages={2481},
  year={2019},
  publisher={IOP Publishing}
}

@article{motreanu2020quasilinear,
  title={Quasilinear Dirichlet problems with competing operators and convection},
  author={Motreanu, Dumitru},
  journal={Open Mathematics},
  volume={18},
  number={1},
  pages={1510--1517},
  year={2020},
  publisher={De Gruyter}
}

@article{crespo2022new,
  title={A new class of double phase variable exponent problems: Existence and uniqueness},
  author={Crespo-Blanco, {\'A}ngel and Gasi{\'n}ski, Leszek and Harjulehto, Petteri and Winkert, Patrick},
  journal={Journal of Differential Equations},
  volume={323},
  pages={182--228},
  year={2022},
  publisher={Elsevier}
}

@article{edmunds2000sobolev,
  title={Sobolev embeddings with variable exponent},
  author={Edmunds, David and R{\'a}kosn{\'\i}k, Ji{\v{r}}{\'\i}},
  journal={Studia Mathematica},
  volume={3},
  number={143},
  pages={267--293},
  year={2000}
}

@article{zhikov1987averaging,
  title={Averaging of functionals of the calculus of variations and elasticity theory},
  author={Zhikov, Vasilii Vasil’evich},
  journal={Mathematics of the USSR-Izvestiya},
  volume={29},
  number={1},
  pages={33},
  year={1987},
  publisher={IOP Publishing}

}

@article{baroni2015harnack,
  title={Harnack inequalities for double phase functionals},
  author={Baroni, Paolo and Colombo, Maria and Mingione, Giuseppe},
  journal={Nonlinear Analysis: Theory, Methods \& Applications},
  volume={121},
  pages={206--222},
  year={2015},
  publisher={Elsevier}
}

@article{baroni2018regularity,
  title={Regularity for general functionals with double phase},
  author={Baroni, Paolo and Colombo, Maria and Mingione, Giuseppe},
  journal={Calculus of Variations and Partial Differential Equations},
  volume={57},
  pages={1--48},
  year={2018},
  publisher={Springer}
}

@article{colombo2015bounded,
  title={Bounded minimisers of double phase variational integrals},
  author={Colombo, Maria and Mingione, Giuseppe and others},
  journal={Arch. Ration. Mech. Anal},
  volume={218},
  number={1},
  pages={219--273},
  year={2015}
}

@article{colombo2015regularity,
  title={Regularity for double phase variational problems},
  author={Colombo, Maria and Mingione, Giuseppe},
  journal={Archive for Rational Mechanics and Analysis},
  volume={215},
  pages={443--496},
  year={2015},
  publisher={Springer}
}

@article{marcellini1991regularity,
  title={Regularity and existence of solutions of elliptic equations with p, q-growth conditions},
  author={Marcellini, Paolo},
  journal={Journal of Differential Equations},
  volume={90},
  number={1},
  pages={1--30},
  year={1991},
  publisher={Academic Press}
}

@article{marcellini1989regularity,
  title={Regularity of minimizers of integrals of the calculus of variations with non standard growth conditions},
  author={Marcellini, Paolo},
  journal={Archive for Rational Mechanics and Analysis},
  volume={105},
  pages={267--284},
  year={1989},
  publisher={Citeseer}
}

@article{avci2019topological,
  title={A topological result for a class of anisotropic difference equations},
  author={Avci, Mustafa},
  journal={Annals of the University of Craiova-Mathematics and Computer Science Series},
  volume={46},
  number={2},
  pages={328--343},
  year={2019}
}

@book{zeidler2013nonlinear,
  title={Nonlinear functional analysis and its applications: II/B: Nonlinear monotone operators},
  author={Zeidler, Eberhard},
  year={2013},
  publisher={Springer Science \& Business Media}
}

@article{gasinski2020existence,
  title={Existence and uniqueness results for double phase problems with convection term},
  author={Gasi{\'n}ski, Leszek and Winkert, Patrick},
  journal={Journal of Differential Equations},
  volume={268},
  number={8},
  pages={4183--4193},
  year={2020},
  publisher={Elsevier}
}

@book{royden2010real,
  title={Real analysis},
  author={Royden, Halsey and Fitzpatrick, Patrick Michael},
  year={2010},
  publisher={China Machine Press}
}

@article{lapa2015no,
  title={No-flux boundary problems involving p (x)-Laplacian-like operators},
  author={Lapa, E Cabanillas and Rivera, V Pardo and Broncano, J Quique},
  journal={Electron. J. Diff. Equ},
  volume={2015},
  number={219},
  pages={1--10},
  year={2015}
}

@article{el2023existence,
  title={Existence result for Neumann problems with p (x)-Laplacian-like operators in generalized Sobolev spaces},
  author={El Ouaarabi, Mohamed and Allalou, Chakir and Melliani, Said},
  journal={Rendiconti del Circolo Matematico di Palermo Series 2},
  volume={72},
  number={2},
  pages={1337--1350},
  year={2023},
  publisher={Springer}
}

@article{albalawi2022gradient,
  title={Gradient and parameter dependent Dirichlet (p (x), q (x))-Laplace type problem},
  author={Albalawi, Kholoud Saad and Alharthi, Nadiyah Hussain and Vetro, Francesca},
  journal={Mathematics},
  volume={10},
  number={8},
  pages={1336},
  year={2022},
  publisher={MDPI}
}

@article{bu2022p,
  title={$(p(x), q(x))$-Kirchhoff-Type Problems Involving Logarithmic Nonlinearity with Variable Exponent and Convection Term},
  author={Bu, Weichun and An, Tianqing and Qian, Deliang and Li, Yingjie},
  journal={Fractal and Fractional},
  volume={6},
  number={5},
  pages={255},
  year={2022},
  publisher={MDPI}
}

@article{galewski2024variational,
  title={On variational competing p, q- Laplacian Dirichlet problem with gradient depending weight},
  author={Galewski, Marek and Motreanu, Dumitru},
  journal={Applied Mathematics Letters},
  volume={148},
  pages={108881},
  year={2024},
  publisher={Elsevier}
}

@book{cruz2013variable,
  title={Variable Lebesgue spaces: Foundations and harmonic analysis},
  author={Cruz-Uribe, David V and Fiorenza, Alberto},
  year={2013},
  publisher={Springer Science \& Business Media}
}

@book{diening2011lebesgue,
  title={Lebesgue and Sobolev spaces with variable exponents},
  author={Diening, Lars and Harjulehto, Petteri and H{\"a}st{\"o}, Peter and Ruzicka, Michael},
  year={2011},
  publisher={Springer}
}

@article{fan2001spaces,
  title={On the spaces {$L^{p(x)}(\Omega)$} and {$W^{m,p(x)}(\Omega)$}},
  author={Fan, Xianling and Zhao, Dun},
  journal={Journal of mathematical analysis and applications},
  volume={263},
  number={2},
  pages={424--446},
  year={2001},
  publisher={Elsevier}
}

@book{radulescu2015partial,
  title={Partial differential equations with variable exponents: variational methods and qualitative analysis},
  author={Radulescu, Vicentiu D and Repovs, Dusan D},
  volume={9},
  year={2015},
  publisher={CRC press}
}

@article{avci2026nehari,
  title={Nehari manifold approach for a singular multi-phase variable exponent problem},
  author={Avci, Mustafa},
  journal={Quaestiones Mathematicae},
  pages={1--21},
  year={2026},
  publisher={Taylor \& Francis}
}

@article{mitidieri2000simple,
  title={A simple approach to Hardy inequalities},
  author={Mitidieri, Enzo},
  journal={Mathematical notes},
  volume={67},
  number={4},
  pages={479--486},
  year={2000},
  publisher={Springer}
}

@book{chipot2009elliptic,
  title={Elliptic equations: an introductory course},
  author={Chipot, Michel},
  year={2009},
  publisher={Springer}
}

\end{document}